\documentclass[11pt,oneside]{amsart}

\usepackage[margin=1in]{geometry}

\usepackage{amsmath,amssymb,amsthm,latexsym,graphicx}
 \usepackage{hyperref}
\newtheorem{theorem}{Theorem}
\newtheorem{corollary}[theorem]{Corollary}

\newtheorem{lemma}[theorem]{Lemma}

\theoremstyle{remark}

\newtheorem{exercise}{Exercise}[section]

\def\ZZ{{\mathbb Z}}

\def\bc{\begin{center}}
\def\ec{\end{center}}
\def\be{\begin{enumerate}}
\def\ee{\end{enumerate}}
\def\bi{\begin{itemize}}
\def\ei{\end{itemize}}
\def\bs{\begin{slide}}
\def\es{\end{slide}}
\def\bx{\begin{exercise}}
\def\ex{\end{exercise}}

\def\l[{\left[}
\def\r]{\right]}
\def\l({\left(}
\def\r){\right)}

\newcommand{\mc}[1]{\ensuremath{\mathcal{#1}}}
\newcommand{\Po}{\mathcal{P}}

\newcommand{\Q}{\mathcal{Q}}

\newcommand{\mon}[1]{\text{mon}(#1)}

\title{Quotient Representations of Uniform Tilings}
\author{Daniel Pellicer}
\author{Gordon Williams}

\begin{document}
%\linenumbers
\begin{abstract}
   Given a flag in each of the vertex-transitive tessellations of the
Euclidean plane by regular polygons, we determine the flag stabilizer under the action of the automorphism group
of a regular cover. In so doing we give a presentation of these tilings as quotients
of regular (infinite) polyhedra.
\end{abstract}
\maketitle
%\doublespacing

\section{Introduction}

The vertex-transitive tessellations of the Euclidean plane have been the object of study for centuries (see, for example, \cite[Section 2.10]{GruShe87}).
There are eleven edge-to-edge vertex-transitive tessellations
of the Euclidean plane by regular convex polygons, up to enantiomorphic forms. Each of them is totally determined by
the cyclic arrangement of polygons around a vertex.
Throughout we follow the notation in Gr\"unbaum and Shepherd \cite{GruShe87}, where $p_{1}.p_{2}.\ldots.p_{n}$ denotes the vertex-transitive tessellation whose
vertices are surrounded by $n$ faces $f_1, \dots, f_n$ (listed in cyclic order) with $f_i$ containing $p_i$ edges. Furthermore, if $p_i = p_{i+1} = \dots = p_{j}$ we may replace
$p_i.p_{i+1}. \cdots .p_{j}$ by $p_i^{j+1-i}$. The tessellations $3^6$, $4^4$ and $6^3$
are regular, both as classical objects, and in the  sense of abstract polyhedra defined below. The remaining eight
tessellations are $3.6.3.6$, $4.8.8$, $3.12.12$, $3.4.6.4$, $3.3.3.4.4$, $3.3.4.3.4$,
$4.6.12$ and $3.3.3.3.6$. As indicated by the notation, these have at least two
different types of tiles.
Throughout this paper we shall refer to these eight (non-regular) tessellations of the plane as the {\em uniform tilings}.

   Abstract polytopes are combinatorial structures satisfying some of the combinatorial
properties of convex polytopes. Of particular interest
are abstract regular polytopes; that is, abstract polytopes that allow all possible
automorphisms given by abstract reflections (see \cite{McMSch02} for details). Michael Hartley \cite{Har99}  proved that every abstract polytope is a quotient of an
abstract regular polytope. This idea was illustrated in \cite{HarWil07} where
presentations of the sporadic Archimedean polyhedra are constructed by finding the minimal regular covers.

   In this paper, we address the problem of presenting the uniform tilings
as quotients of regular polyhedra. For each tiling, we determine an enumerable generating set
for the stabilizer of a flag under the flag action from a string C-group. Furthermore, we prove that
such stabilizer contains no finite generating set. The problem of determining
the minimal regular covers is beyond the scope of the current work and will be
discussed in subsequent articles.
%We discuss some properties of the regular covers,
%with special attention given to those that arise as a consequence of the fact that uniform tilings have infinitely many flags (triples consisting of incident vertex, edge and face).

%%%%%%%%%%%%%%%%%%%%%%%%%%%%%%%%%%

   We begin with some preliminary material, referring to \cite{McMSch02}
and \cite{HarWil07} for details.

\section{Abstract Polyhedra and Related Objects}
Following \cite[Section 2A]{McMSch02}, we define   an {\em abstract $d$-polytope} $\Po$ to be a partially ordered set whose elements are
called {\em faces}, with partial order denoted by $\le$, and that satisfies the following properties. It contains a minimum
face $F_{-1}$ and maximum face $F_d$, and all maximal totally
ordered subsets of $\Po$, the {\em flags} of $\Po$, contain precisely $d+2$ elements including $F_{-1}$ and
$F_d$. Consequently, $\le$ induces a strictly increasing rank function such that the ranks
of $F_{-1}$ and $F_d$ are $-1$ and $d$ respectively. Finally, $\Po$ is strongly
connected and satisfies the ``diamond condition'' (see \cite[Section 2A]{McMSch02} for details).

   In the present paper we are interested only in {\em abstract polyhedra}, that is, abstract
polytopes of rank $3$; however Theorem \ref{T:SpanningTrees} and Corollary \ref{T:finitePolytopes} have relevance to abstract polytopes of general rank. Throughout the remainder of this paper we will use ``polyhedra'' to mean either the geometric objects or abstract polyhedra, as appropriate. The {\em vertices} and {\em edges} of an abstract polyhedron are its
faces of rank $0$ and $1$ respectively. In this context there is little possibility
of confusion if we refer to the rank $2$ faces simply by {\em faces}. We define a {\em section} $F/G$ of a polytope to be the collection of all faces $H$ such that $G\le H\le F$. The {\em vertex-figure}
at a vertex $v$ is the section $\{F \in \mc P\, |\, v \le F\}$. In the case of
polyhedra, the diamond condition requires that every edge contains precisely two vertices
and is contained in precisely two faces, and for any vertex $v$ contained
in a face $f$ there are precisely two edges containing $v$ which are contained in $f$.
As a consequence of the diamond condition, given $i \in \{0, 1, 2\}$ and a flag $\Psi$,
there exists a unique flag $\Psi^i$ that coincides with $\Psi$ in all faces except
in the face of rank $i$. The flag $\Psi^i$ is called the {\em i-adjacent flag} of
$\Psi$. The strong connectivity for polyhedra implies that every face and every
vertex-figure is isomorphic to a polygon, that is, a cycle in the
graph theoretic sense. The {\em degree} of a vertex $v$ is the number of edges
containing $v$, and the {\em co-degree} of a face $f$ is the number of edges contained
in $f$. %For convenience we will consider the flags of a polyhedron as triplets $(v, e, f)$ that determine them, \daniel where $v$, $e$ and $f$ are incident vertex, edge and face, respectively.

%   The convex polyhedra and the uniform tilings are examples of abstract polyhedra.

   Whenever every vertex of a polyhedron $\Po$ has the same degree $p$, and every face of
$\Po$ has the same co-degree $q$ we say that $\Po$ is {\em equivelar} and has {\em Schl\"afli
type} $\{p, q\}$.

   An {\em automorphism} of a polyhedron $\Po$ is an order preserving bijection of its elements. We say that a polyhedron is {\em regular} if its automorphism group $\Gamma(\Po)$ is transitive on the set of flags of \mc P, which we will denote by $\mc F(\mc P)$. The Platonic solids and the tessellations $3^6$, $4^4$ and $6^3$ are examples of abstract regular polyhedra.

   A {\em string C-group} $G$ of rank $3$ is a group with distinguished
involutory generators $\rho_0, \rho_1, \rho_2$, where 
$(\rho_0 \rho_2)^2 = \varepsilon$, the identity in $G$, and $\langle \rho_0, \rho_1 \rangle \cap \langle \rho_1, \rho_2 \rangle
= \langle \rho_1 \rangle$  (this is called the {\em intersection condition}). The automorphism group of an abstract regular polyhedron is always a string
C-group of rank $3$. Having fixed an arbitrarily chosen {\em base flag} $\Phi$, we obtain $\rho_{i}$ as the (unique) automorphism mapping $\Phi$ to the $i$-adjacent flag $\Phi^{i}$.
Furthermore, any string C-group of rank $3$ is the automorphism
group of an abstract regular
polyhedron \cite{McMSch02}, so, up to isomorphism, there is a one-to-one correspondence between the string C-goups of rank $3$
and the abstract regular polyhedra.  Thus, in the study of abstract polyhedra we may either  work with the polyhedron as a poset, or with its automorphism group. Automorphism groups of regular polyhedra will be denoted by $\Gamma$ in this paper.

   For any polyhedron $\Po$ we define permutations $r_0, r_1, r_2$ on $\mc F(\mc P)$ by
\[\Psi{r_i} := \Psi^i,\]
for every flag $\Psi$ of $\Po$ and $i=0, 1, 2$ (note that these are {\em not} automorphisms of \mc P). The group $\mon{\mc P}:=\langle r_0, r_1, r_2 \rangle$
will be referred to as the {\em monodromy group} of $\Po$ (see \cite{HubOrbWei},
 but note that this definition differs from the definition in \cite{Zvo98}, where the author only considers words with even length in the generators $r_i$). The {\em flag action} of a string C-group $\Gamma=\langle \rho_0, \rho_1, \rho_2 \rangle$ on $\Po$ is the group homomorphism
$\Gamma\to \mon{\mc P}$ defined by $\rho_i\mapsto r_i$, provided such a homomorphism exists. In this context, if $w = w' \rho_i$
for some $w' \in \Gamma$ then $\Psi^w  = (\Psi^{w'}){r_i}=(\Psi^{w'})^{i}$. Note that, by definition of automorphism, the action of each $r_i$ (and thus the flag action) commutes with the automorphisms of any
given polyhedron. That is,
\begin{equation}\label{eq:actionsymmetry}
(\Psi{r_i}) \alpha = (\Psi \alpha){r_i}
\end{equation}
for $i=0, 1, 2$ and $\alpha \in \Gamma(\Po)$.

We say that the regular polytope $\mc P$ is a {\em cover} of $\mc Q$ if 
$\mc Q$ admits a flag action from $\Gamma(\mc P)$, such a cover is denoted by $\mc P\searrow \mc Q$. (This implies the notion of covering described in \cite[p. 43]{McMSch02}.) For example, the (universal) polyhedron with automorphism group isomorphic to the
{\em Coxeter group}
$[\infty, \infty] := \langle \rho_0, \rho_1, \rho_2 \,|\, (\rho_0 \rho_2)^2 = \varepsilon \rangle$ covers all other polyhedra. Whenever the least common multiple of the co-degrees
of the faces of a polyhedron $\mc P$ is $p$, and the least common multiple of the vertex degree  of $\mc P$ is $q$, $\mc P$ is covered by the
tessellation $\{p, q\}$ whose automorphism group is isomorphic to the Coxeter group
\[[p, q] := \langle \rho_0, \rho_1, \rho_2 \,|\, (\rho_0 \rho_2)^2 =
(\rho_0 \rho_1)^p = (\rho_1 \rho_2)^q = \varepsilon \rangle.\]
(Recall that $\{p,q\}$ can be viewed as a regular tessellation of the sphere, Euclidean plane or hyperbolic plane, according as $\frac{1}{p}+\frac{1}{q}>\frac{1}{2}$, $=\frac{1}{2}$ or $<\frac{1}{2}$, respectively.)   

Whenever $\mc P\searrow \mc Q$, we find that  $\mc Q$ is totally
determined by $\mc P$ and the stabilizer $N$ of a chosen base flag $\Phi$ of $\mc Q$ under the flag
action of $\Gamma(\mc P)$. Indeed,  $\mc Q = \mc P / N$, the polytope whose faces are orbits under the action of $N$ on $\mc P$. For further details, refer to \cite{Har99}.

   For a given abstract polyhedron $\mc P$, we define its flag graph $\mc{GF}(\mc P)$ as the
edge-labeled graph whose vertex set consists of all flags of $\Po$, where two vertices (flags) are joined by an edge
 if and only if they are $i$-adjacent for some $i=0, 1, 2$. We label each edge with with $i$ according to the $i$-adjacency determining the edge; e.g., if $\Psi^{0}=\Upsilon$, then the edge connecting $\Psi$ and $\Upsilon$ is labeled with a 0.
\section{The Structure of Stabilizers for Tilings}

In this section we shall use the flag graph of a uniform tiling $\Q$ to determine the stabilizer of a given flag
of $\Q$ under the flag action of the automorphism group of a regular cover $\Po$. We begin with some basic graph theoretical definitions.

We define a {\em walk} in the flag graph of $\mc Q$ to be a sequence of vertices of $\mc{GF}(\mc Q)$ (that is, flags of $\mc Q$) $\alpha=(\Psi_{0},\Psi_{1},\ldots,\Psi_{n})$ (possibly infinite) such that $\Psi_{i}, \Psi_{i+1}$ share an edge; if all the vertices are distinct then we say that $\alpha$ is a {\em path}. We define $|\alpha|$, the {\em length} of $\alpha$, to be $n$.  We will use juxtaposition to denote the concatenation of walks, so if $\beta=(\Psi_{n},\Psi_{n+1},\ldots,\Psi_{n+m})$, then $\alpha\beta=(\Psi_{0},\Psi_{1},\ldots,\Psi_{n},\Psi_{n+1},\ldots,\Psi_{n+m})$ has $|\alpha\beta|=|\alpha|+|\beta|$. If $\alpha=(\Psi_{0},\Psi_{1},\ldots,\Psi_{n})$ is a walk in $\mc {GF}(\mc Q)$ we define an associated word $w_{\alpha}=\rho_{i_{0}}\rho_{i_{1}}...\rho_{i_{n-1}}$ in the generators of an associated string C-group, where $\Psi_{j+1}=\Psi_{j}^{\rho_{i_{j}}}$. Conversely, given a flag $\Phi$ of $\Q$, any word $w=\rho_{i_{0}}\cdots \rho_{i_{n-1}}$ on the generators of $\Gamma$, determines in a natural way  the walk $\alpha_{w}=(\Phi=\Psi_{0},\ldots,\Psi_{n})$ in $\mc{GF}(\mc Q)$ where $\Psi_{j+1}=\Psi_{j}^{\rho_{i_{j}}}$.

Suppose we have a walk of the form $\beta=(\Psi_{0},\Psi_{1},...,\Psi_{k-1},\Psi_{k},\Psi_{k-1},...,\Psi_{0})$, and let $\alpha=(\Psi_{0},\Psi_{1},...,\Psi_{k-1},\Psi_{k})$, then $\beta$ is the walk obtained from $\Psi$ by the word $w_{\alpha}w_{\alpha}^{-1}=\varepsilon$, which maps to the trivial word in the monodromy group. We may insert or delete such strings in walks at will; we say that two walks that differ only by such redundant terms are equivalent and denote this relation by $\sim$.

Given a polytope \mc Q with regular cover \mc P whose automorphism group is $\Gamma$, to construct a representation of \mc Q as a quotient of \mc P, it is necessary to identify $N=Stab_{\Gamma}(\Phi)$, where $\Phi$ is a specified {\em base flag} of \mc Q. Throughout the discussion that follows, any walk determined by a word $w\in\Gamma$ will be assumed to start at $\Phi$. Let $T$ be a spanning tree for $\mc{GF}(\mc Q)$. For a given (oriented) edge $e=(\Psi,\Upsilon)\in\mc{GF}(\mc Q)$, we define $\beta_{e}$ to be the concatenation of the unique walk $x\in T$ from $\Phi$ to $\Psi$ with $e$ and the unique walk $y\in T$ from $\Upsilon$ to $\Phi$.   An essential tool in identifying generators for $N$ is the complement of a spanning tree (tree containing all vertices) in the flag graph of \mc Q, as seen below.

\begin{theorem}\label{T:SpanningTrees} Let $T$ be a spanning tree in $\mc{GF}(\mc Q)$ rooted at $\Phi$, a specified (base) flag of \mc Q. For each edge $e=(\Psi,\Upsilon)$ of $\mc{GF}(\mc Q)$, define the unique walk $\beta_{e}$ as above. %(note that for the purposes of this theorem, we have one such walk for each orientation of the determining edges).
Then $S=\{w_{\beta_{e}}:e\in\mc{GF}(\mc Q)\setminus T\}$ is a generating set for $Stab_{\Gamma}(\Phi)$.\label{t:treeGenerates}

%\begin{theorem}\label{T:SpanningTrees} Let $T$ be a spanning tree in $\mc{GF}(\mc Q)$ rooted at $\Phi$, a specified (base) flag of \mc Q. For each edge $e=(\Phi_{i},\Phi_{j})$ of $\mc{GF}(\mc Q)$, define the unique walk $\beta_{e}=(\Phi=\Psi_{0},\Psi_{1},...,\Psi_{k}=\Phi_{i},\Psi_{k+1}=\Phi_{j},\Psi_{k+2},\ldots,\Psi_{n}=\Phi)$ where $(\Phi=\Psi_{0},\Psi_{1},...,\Psi_{k}=\Phi_{i})$ and $(\Psi_{k+1}=\Phi_{j},\Psi_{k+2},\ldots,\Psi_{n}=\Phi)$ are paths in $T$. Then $S=\{w_{\beta_{e}}:e\in\mc{GF}(\mc Q)\setminus T\}$ is a generating set for $Stab_{\Gamma}(\Phi)$.\label{t:treeGenerates}

\begin{proof} First, we note that it follows easily from the axioms that \mc Q is at most a countable set, and so $\mc{GF}(\mc Q)$ is a finite or countable graph, implying that the requisite spanning tree $T$ exists. Second,  it is worth noting that $\beta_{e}$ is well defined because there is exactly one path connecting any two vertices of $\mc{GF}(\mc Q)$ in $T$. Third, we observe that any walk corresponding to an element of the stabilizer of $\Phi$ must be closed, so if $w\in Stab_{\Gamma}(\Phi)$, then $\alpha_{w}$ starts and ends at $\Phi$.

It suffices then to show that for any element $w\in Stab_{\Gamma}(\Phi)$, that $\alpha_{w}$ may be obtained as a union of walks of the form $\beta_{e}$: that is, $\alpha_{w}\sim \beta_{e_{1}}\cdots\beta_{e_{k}}$.

We will proceed by induction on $n_{\gamma}$, the number of times  a closed walk $\gamma$ starting and ending at $\Phi$ traverses edges of $\mc{GF}(\mc Q)\setminus T$ . If $n_{\gamma}=0$ then $\gamma$ lies entirely in $T$, and so the corresponding word in the generators of $\Gamma$ is trivial (that is, reduces to $\varepsilon$). If $n_{\gamma}=1$ then $\gamma$ contains only a single edge $e_{1}\in E(\mc{GF}(\mc Q)\setminus T)$. Thus the remainder of $\gamma$ is in $T$, and so is unique (up to equivalence). Thus $\gamma=\beta_{e_{1}}$.

Suppose now that we have shown that for any closed walk $\gamma$ containing up to $k$ edges of $\mc{GF}(\mc Q)\setminus T$, $\gamma$ may be written as a concatenation of corresponding closed walks $\beta_{e_{1}},\ldots,\beta_{e_{k}}$. Let $\delta$ be a closed walk at $\Phi$ containing $(k+1)$ edges of $\mc{GF}(\mc Q)\setminus T$, and denote them $e_{1},\ldots,e_{k+1}$ in the order they are traversed by $\delta$. Let $\Psi_{i,1}$ and $\Psi_{i,2}$ denote the vertices---in the order traversed---of the edge $e_{i}$. Let  $\tau$ be the unique path in $T$ connecting $\Phi$ and $\Psi_{k+1,1}$ (it is possible $e_{k}$ and $e_{k+1}$ share a vertex). Denote by $\delta_{1}$ the portion of the walk $\delta$ connecting $\Phi$ to $\Psi_{k+1,1}$ containing the edges $e_{1},\ldots, e_{k}$ and by $\delta_{2}$ the portion of the walk $\delta$ connecting  $\Psi_{k+1,1}$ and $\Phi$ containing the edge $e_{k+1}$. Then $\delta_{{1}} \tau^{-1}$ is a closed walk at $\Phi$ containing $k$ edges of $\mc{GF}(\mc Q)\setminus T$ and $\tau\delta_{2}$ is a closed walk at $\Phi$ containing 1 edge of $\mc{GF}(\mc Q)\setminus T$, and $\delta_{1}\tau^{-1}\tau\delta_{2}\sim\delta$. By the strong inductive hypothesis, $\delta_{1}\tau^{-1}\sim \beta_{e_{1}}\beta_{e_{2}}\cdots\beta_{e_{k}}$ and $\tau\delta_{2}\sim\beta_{e_{k+1}}$. Thus $\delta\sim\beta_{e_{1}}\beta_{e_{2}}\cdots\beta_{e_{k+1}}$, completing the induction.

\end{proof}

\end{theorem}

\begin{corollary}\label{T:finitePolytopes} Any finite polytope \mc Q  admits a quotient presentation $\Gamma/N$ in which $N$ has finitely many generators.\end{corollary}
\begin{proof} It suffices to observe that any tree in $\mc{GF}(\mc Q)$ omits a finite number of edges.\end{proof}

\begin{theorem}\label{T:UniformNotFinite} Let \mc Q be a uniform  tiling of the plane and $\Phi$ a specified base flag in \mc Q. Then $Stab_{\Gamma}(\Phi)$ has no finite generating set of words in the generators of $\Gamma$.
\begin{proof}
Define the distance $d(\Upsilon,\Psi)$ between two flags $\Upsilon$ and $\Psi$  to be the length of the shortest path connecting those two flags in the flag graph. In particular, to each flag $\Psi$ of $\mc Q$ we may associate its distance $d_{\Psi}$ to the base flag $\Phi$.

Suppose, for the sake of contradiction, that $S=\{w_{1},\ldots,w_{k}\}$ is a finite set of words in the generators of $\Gamma$ that generates $Stab_{\Gamma}(\Phi)$. Note then that each $w_{i}$ determines a unique walk $\alpha_{w_{i}}$ that starts and ends at $\Phi$. In particular, the length of each of these walks is finite. Also observe, that the product of any of the elements of $S$ will correspond to a concatenation of the walks $\{\alpha_{w_{1}},\ldots,\alpha_{w_{k}}\}$. In particular, no product of the elements of $S$ or their inverses will yield a walk starting at $\Phi$ that traverses an edge that does not belong to one of the $\alpha_{w_{i}}$.

Let $d_{i}=\displaystyle\max_{\Psi\in\alpha_{w_{i}}}d_{\Psi}$, and $d=\max_{i}d_{i}$; then $d$ measures the greatest distance between $\Phi$ and any flag in one of the $\alpha_{w_{i}}$. Note that $d<\infty$ since the set of vertices in all the walks  $\{\alpha_{w_{1}},\ldots,\alpha_{w_{k}}\}$ is finite.  

Since \mc Q is a  uniform tiling of the plane, there exist integers $m$ and $n$ where $n$ is the degree at each vertex, and $m$ is divisible by the number of sides of each of the faces of \mc Q.  Then without loss of generality we may assume that the polyhedron with automorphism group $\Gamma$ is of type $\{m,n\}$. Also note that since \mc Q is infinite, there exists a vertex $v$ of \mc Q such that for any flag $\Psi$ containing $v$, $d_{\Psi}>d$. Let $f$ be a face of $\mc Q$ containing $v$ with $q$ sides such that $m/q\ne 1$: note that such a face must exist since \mc Q is not regular. Let $T$ be a spanning tree of $\mc{GF}(\mc Q)$ and  let $\Upsilon$ be a flag of \mc Q containing $v$ and $f$. Let $\alpha_{\Upsilon}$ be the unique path connecting $\Phi$ and $\Upsilon$ in $T$, and let $w_{\alpha_{\Upsilon}}$ be the corresponding word in $\Gamma$. We now observe that $\sigma=w_{\alpha_{\Upsilon}}(\rho_{0}\rho_{1})^{q}w_{\alpha_{\Upsilon}}^{-1}$ is a nontrivial element of $Stab_{\Gamma}(\Phi)$ (since $(\rho_{0}\rho_{1})^{q}$ is nontrivial in $\Gamma$). Moreover, since $d_{\Upsilon}>d$, $\sigma\notin\langle S\rangle$, contradicting our initial assumption and therefore no finite set generates $Stab_{\Gamma}(\Phi)$.
\end{proof}
\end{theorem}
Note, however, that the conclusion of this theorem is decidedly different than in the case of a regular tiling of the plane. For example, if we consider the regular tiling of the plane $\mc R =\{3,6\}$ by triangles, the generating set  for the stabilizer of a specified base flag is precisely the defining relations for \mc R. Specifically, $\Gamma(\mc R)=\langle s_{0},s_{1},s_{2}\rangle/\langle s_{0}^{2},s_{1}^{2},s_{2}^{2},(s_{0}s_{1})^{3},(s_{0}s_{2})^{2},(s_{1},s_{2})^{6}\rangle$, and so $\{s_{0}^{2},s_{1}^{2},s_{2}^{2},(s_{0}s_{1})^{3},(s_{0}s_{2})^{2},(s_{1},s_{2})^{6}\}$ forms a finite generating set for $Stab_{\Gamma}(\Phi)$.

To demonstrate that the algorithms for producing the generators listed in the next section suffice, we require Theorem \ref{t:VFgenerate} to establish that generators corresponding to walks from the base flag to (and around) each face and vertex are sufficient. Lemma \ref{t:vertexLemma} demonstrates that only one such generator for each vertex or face of the polyhedron is necessary. These results allow us to easily find an enumerable set of generators.

\begin{theorem} Let the polyhedron \mc Q be a map on the sphere or the Euclidean plane, $\Phi$ the base flag of \mc Q, and $\Gamma$ a string $C$-group with generators $\rho_{0},\rho_{1},\rho_{2}$ and a flag action on \mc Q. Then $Stab_{\Gamma}(\Phi)$ is generated by the set of elements $$W_{v}=w_{v}^{-1}(\rho_{1}\rho_{2})^{q_{v}}w_{v} \text{ and }W_{f}=w_{f}^{-1}(\rho_{0}\rho_{1})^{p_{f}}w_{f},$$ where $v$ is any vertex of \mc Q of degree $q_{v}$, and $f$ is any face of $\mc Q$ with $p_{f}$ edges, and $w_{v}$ and $w_{f}$ are words which map $\Phi$ to a flag containing $v$ or $f$, respectively.
\label{t:VFgenerate}
\begin{proof}    Since $\Gamma$ has a flag action on \mc Q,  \mc Q is a quotient of the polytope $\mc P=\mc P(\Gamma)$. Also, $\mc{GF}(\mc Q)$ has a natural plane embedding. Therefore it makes sense to consider the cells of $\mc{GF}(\mc Q)$. We say that a walk {\em encloses} a cell if the winding number of the walk about any point in that cell is not zero.

Clearly, all the elements $W_{v}$ and $W_{f}$ belong
to $Stab_{\Gamma}(\Phi)$, since each of these fixes the base flag.  One useful observation in what follows is that  $w_{z}(\rho_{0}\rho_{2})^{2}w_{z}^{-1}$ (i.e., a walk out to, and then around an edge of the polyhedron and back) is always trivial since $(\rho_{0}\rho_{2})^{2}$ is trivial in the covering group.
On the other hand, any element in $Stab_{\Gamma}(\Phi)$ corresponds to a closed walk in the flag
graph of $\mc Q$ starting and ending at $\Phi$. 

Suppose, for the sake of contradiction,  that there are nontrivial elements in $Stab_{\Gamma}(\Phi)$ which are not generated by
the elements stated in the theorem. In particular, each of the corresponding walks must enclose at least two cells. Among these elements, we consider all those where the number
of cells enclosed by the corresponding walks in the flag graph is minimal. Among the latter elements, we choose a particular $w_0$ which has minimal length as a word on the generators of $\Gamma$.  That is, in the flag graph, we are identifying a walk $\alpha_{w_{0}}$ with a minimal number of edges that cannot be expressed as a concatenation of walks associated with the $W_{v}$ and $W_{f}$. By construction, $\alpha_{w_{0}}$ encloses a connected region; since it is chosen to have a minimal number of edges, it cannot wind around the connected region more than once.

Let $\Psi$ be the first vertex of $\mc{GF}(\mc Q)$ (i.e., a flag of \mc Q) other than $\Phi$ that appears at least twice in the walk $\alpha_{w_{0}}$. (If no such $\Psi$ exists, we may move immediately to the last part of the proof.) Determine walks $\alpha_{1},\alpha_{2},\alpha_{3}$ so that $\alpha_{w_{0}}=\alpha_{1}\alpha_{2}\alpha_{3}$, $\alpha_{1}$ is a walk from $\Phi$ to $\Psi$, $\alpha_{2}$ is a closed walk at $\Psi$ and $\alpha_{3}$ is a walk from $\Psi$ to $\Phi$. Moreover, we require that $\Psi$ does not appear in any edge of $\alpha_{1}$ or $\alpha_{3}$. Let $\widehat{\alpha}=\alpha_{1}\alpha_{2}\alpha_{1}^{-1}$ and $\tilde \alpha =\alpha_{1}\alpha_{3}$; then $\alpha_{w_{0}}\sim\widehat\alpha \tilde \alpha$. Since $\tilde \alpha$ is shorter than $\alpha_{w_{0}}$, the word corresponding to $\tilde \alpha$ must be generated by the elements of the form  $W_{v}$ and $W_{f}$. Since $\alpha_{w_{0}}$ is not generated by such elements, $\widehat \alpha$ must not be either.

Since $\alpha_{w_{0}}$ is the shortest such walk under consideration, $|\alpha_{w_{0}}|\le|\widehat \alpha|$, so that
\begin{align*}
|\alpha_{w_{0}}|=|\alpha_{1}|+|\alpha_{2}|+|\alpha_{3}|& \le |\widehat \alpha|=2|\alpha_{1}|+|\alpha_{2}|,
\end{align*}
which forces  $|\alpha_{3}|\le |\alpha_{1}|.$
The same analysis done to $\alpha_{w_{0}}^{-1}$ instead of $\alpha_{w_0}$ implies that
\begin{align*}
|\alpha_{1}|=|\alpha_{1}^{-1}|& \le |\alpha_{3}^{-1}|=|\alpha_{3}|,
\end{align*}
so that $|\alpha_1 |= |\alpha_3|$.
%Similarly, if we let $\ol{\alpha}=\alpha_{3}^{-1}\alpha_{2}\alpha_{3}$, since $\alpha_{w_{0}}=\tilde \alpha\ol{\alpha}$, then $|\alpha_{w_{0}}|\le|{\ol{\alpha}}|$ and so
%\begin{align*}
%|\alpha_{w_{0}}|=|\alpha_{1}|+|\alpha_{2}|+|\alpha_{3}|& \le |\ol{{\alpha}}|=|\alpha_{2}|+2|\alpha_{3}|\end{align*}
%which forces $|\alpha_{1}|\le |\alpha_{3}|$, and so $|\alpha_{1}|=|\alpha_{3}|$.

We claim now that $\tilde \alpha$ encloses no cells, i.e., $\alpha_{3}=\alpha_{1}^{-1}$.  Note that $\widehat{\alpha}=\alpha_{1}\alpha_{2}\alpha_{1}^{-1}$ has the same length as $\alpha_{w_{0}}$, %does not contain any of the cells enclosed by $\tilde\alpha$, 
and corresponds to a word that is not generated by elements $W_{v}$ and $W_{f}$. Since $\alpha_{w_{0}}$ enclosed the least number of cells of any such walk, and $\widehat{\alpha}$ can't enclose any cell not enclosed by $\alpha_{w_{0}}$, $\tilde\alpha$ must not enclose any cells at all because $\alpha_{w_{0}}\sim\widehat{\alpha}\tilde{\alpha}$. Thus $\alpha_{3}=\alpha_{1}^{-1}$. This means that $\Psi$ must have been adjacent to $\Phi$ and that $\alpha_{1}$ is a single edge of $\mc {GF}(\mc Q)$.  We can repeat this argument on successive repeated flags and conclude therefore that $\alpha_{w_{0}}$ looks like the closed walk in Figure \ref{generatorcell}(a), that is, a cycle with a tail starting at $\Phi$.

\begin{figure}
\begin{center}
\includegraphics[width=10cm, height=4.5cm]{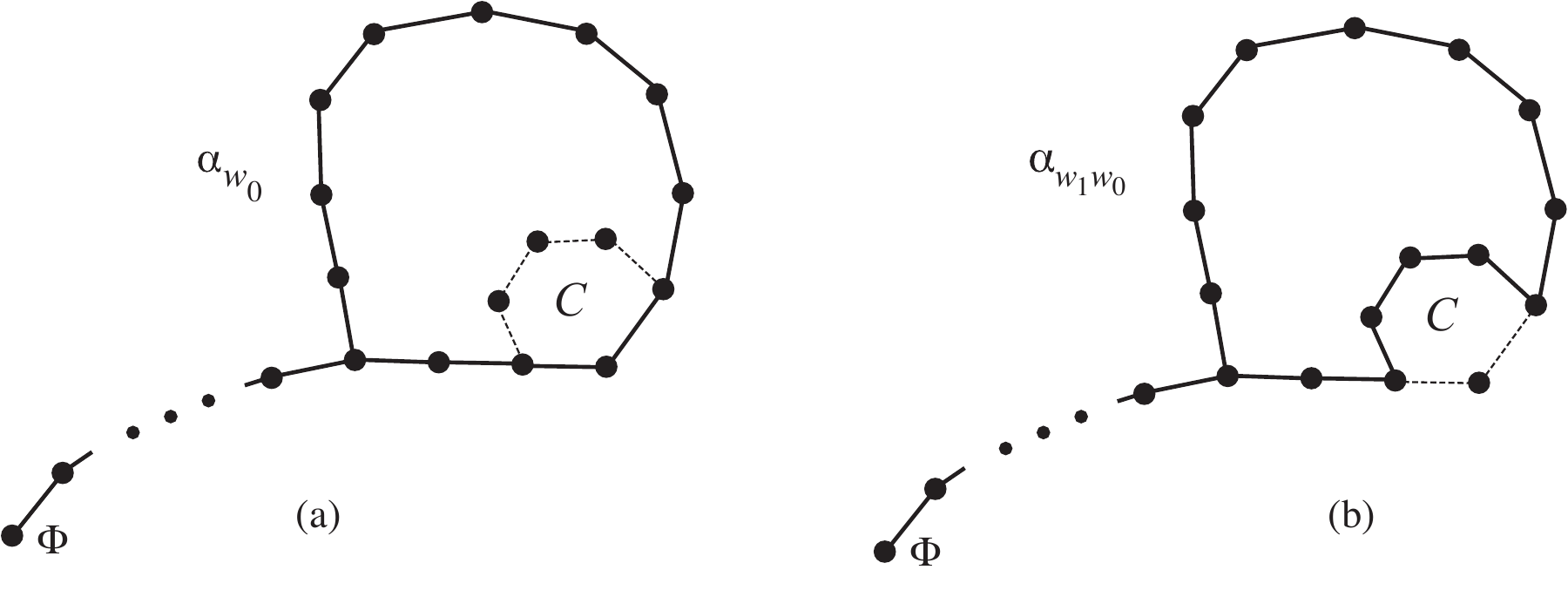}
\caption{The walks $\alpha_{w_{0}}$ and $\alpha_{w_{1}w_{0}}$ indicated with bold edges.}
\label{generatorcell}\end{center}
\end{figure}

Finally, consider a cell $C$  of the flag graph enclosed by, and sharing an edge with, $\alpha_{w_{0}}$.  Note that all cells of the flag graph are even cycles. 
Assume that $C$ is a $q$-cycle with edges of alternating labels $i$ and $j$, let $z$ be a vertex of $C$ which belongs also to $\alpha_{w_{0}}$, and let $w_z$ be the that part of the word $w_0$ corresponding to the walk from the initial vertex  $\Phi$ to $z$. There are now two cases to consider. If $q=4$, then the edge labels for the cycle must be 0 and 2, which contradicts the cell enclosing minimality of $\alpha_{w_{0}}$ since  $w_{z}(\rho_{0}\rho_{2})^{2}w_{z}^{-1}$ is trivial in $\Gamma$ and so $C$ could have been removed from $\alpha_{w_{0}}$. If $q>4$ 
then let $w_1 = w_z (\rho_i \rho_j)^{q/2} w_z^{-1} $. Observe that, by construction of $w_0$, the set of cells enclosed by $w_1$ most be non-empty. In particular, $C$ is not enclosed oppositely by $w_0$ and $w_1$. Note also that $w_{1}\in Stab_{\Gamma}(\Phi)$, so therefore $w_1 w_0 \in Stab_{\Gamma}(\Phi)$ also. By construction, the walk $\alpha_{w_{1}w_{0}}$ encloses all cells enclosed by $\alpha_{w_0}$ except $C$ (see Figure \ref{generatorcell}(b)). Thus, by hypothesis, $w_1 w_0$ is generated by the elements $W_{v}$ and $W_{f}$ because $\alpha_{w_{1}w_{0}}$ encloses fewer cells than $\alpha_{w_{0}}$. Hence,  $w_0$ must also be generated by elements of the form $W_{v},W_{f}$, which contradicts our initial supposition.

\end{proof}
\end{theorem}

It is worth noting that this theorem does not necessarily hold for abstract polyhedra that admit presentations as maps on the projective plane or surfaces of higher genus, because the notion of winding number is not well defined in these settings. For example, on the hemi-octahedron $\rho_{1}\rho_{2}\rho_{1}\rho_{0}(\rho_{1}\rho_{2})^{2}\rho_{0}$ (the antipodal map on the octahedron) fixes the bases flag but is not generated by elements of type $W_{v}$ or $W_{f}$.

To demonstrate that a single generator of the type $W_{v}$ or $W_{f}$ for each vertex $v$ or face $f$ suffices to generate $Stab_{\Gamma}(\Phi)$, we must demonstrate that given two walks $w$ and $w'$ to the cell determined by the vertex $v$ or the face $f$,  and given  $w(\rho_{i}\rho_{i+1})^{q}w^{-1}\in Stab_{\Gamma}(\Phi)$, then $w'(\rho_{i}\rho_{i+1})^{q}w'^{-1}$ is automatically in $Stab_{\Gamma}(\Phi)$ as well.

\begin{lemma}\label{l:notallofthem}
   Let $\mc Q$ be a polyhedron, $\Phi$ a flag of $\mc Q$, and $\Gamma$
a string C-group with generators $\rho_0, \rho_1, \rho_2$ and flag action on $\mc Q$.
If $w (\rho_i \rho_{i+1})^q w^{-1} \in Stab_{\Gamma}(\Phi)$ then $w' (\rho_i \rho_{i+1})^q w'^{-1} \in Stab_{\Gamma}(\Phi)$
for any $w'$ such that $\Phi w'$ and $\Phi w$ coincide in their face if $i=0$, and in their vertex if $i=1$. \label{t:vertexLemma}
\end{lemma}

\begin{figure}
\begin{center}
\includegraphics[width=10cm, height=4.5cm]{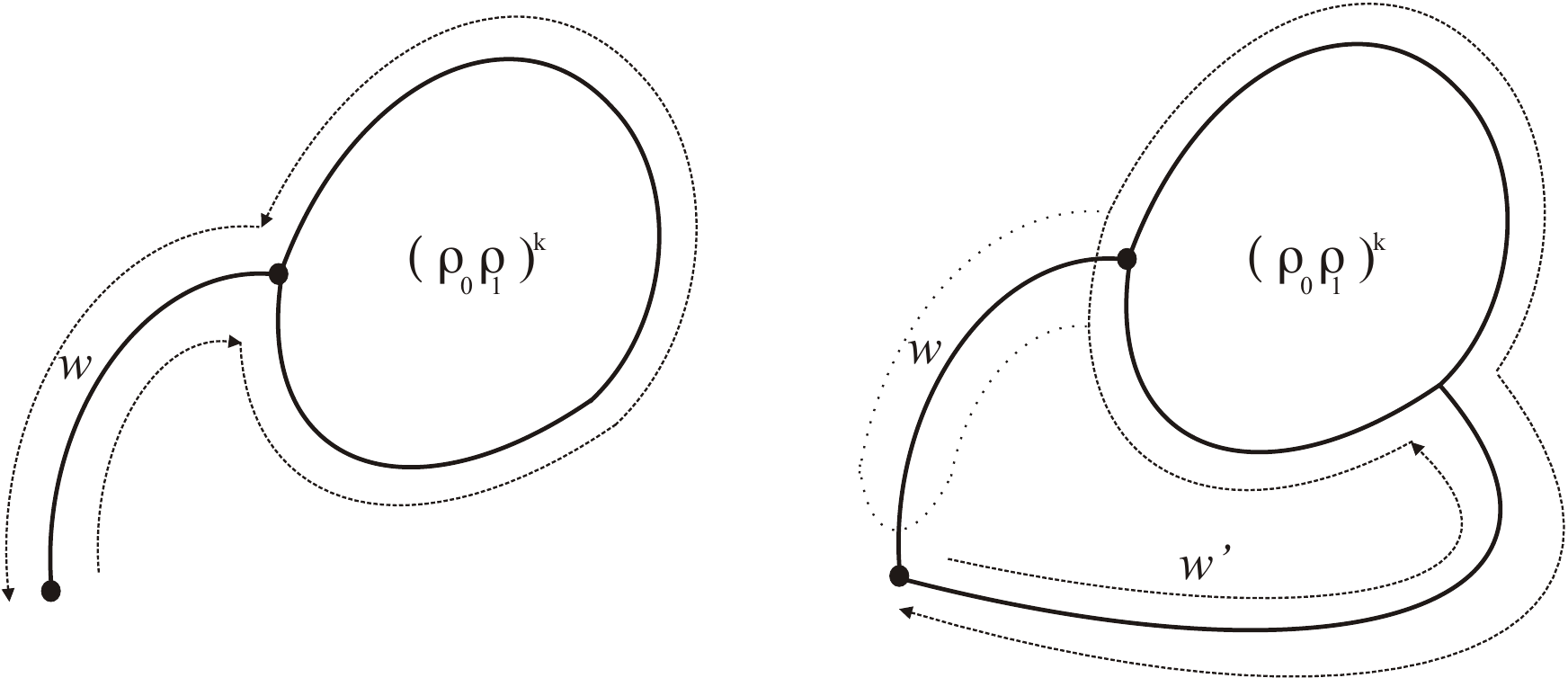}
\caption{One generator of $Stab_{\Gamma}(\Phi)$ induces the other}\label{generatorstab}
\end{center}
\end{figure}

\begin{proof}
We will prove the lemma in the case where $i=0$, and $\Phi w$ and $\Phi w'$ coincide in their face. The identical argument holds for $i=1$, when $\Phi w$ and $\Phi w'$ coincide in their vertex.

   Since the face of $\Phi w'$ coincides with the face of $\Phi w$, there exists $x \in \langle \rho_0, \rho_1 \rangle$
such that $\Phi w = \Phi w' x$. If $x = (\rho_0 \rho_1)^k$ for some integer $k$ then
\begin{eqnarray*}
\Phi w' (\rho_0 \rho_1)^q w'^{-1} &=& \Phi w' (\rho_0 \rho_1)^k  (\rho_0 \rho_1)^{-k} (\rho_0 \rho_1)^q w'^{-1}\\
&=& \Phi w (\rho_0 \rho_1)^{-k} (\rho_0 \rho_1)^q w'^{-1}\\
&=& \Phi w (\rho_0 \rho_1)^q (\rho_0 \rho_1)^{-k} w'^{-1}\\
&=& \Phi w (\rho_0 \rho_1)^{-k} w'^{-1}\\
&=& \Phi w' w'^{-1} = \Phi.
\end{eqnarray*}
A similar computation for the case $x = \rho_1 (\rho_0 \rho_1)^k$ concludes the argument.
%
%On the other hand, for $i=1$
%\begin{eqnarray*}
%\Phi w' (r_0 r_1)^q w'^{-1} &=& \Phi w' r_1(r_0 r_1)^k  (r_0 r_1)^{-k}r_1 (r_0 r_1)^q w'^{-1}\\
%&=& \Phi w (r_0 r_1)^{-k} r_1 (r_0 r_1)^q w'^{-1}\\
%&=& \Phi w (r_0 r_1)^{-k} (r_0 r_1)^{-q} r_1 w'^{-1}\\
%&=& \Phi w (r_0 r_1)^{-q} (r_0 r_1)^{-k} r_1 w'^{-1}\\
%&=& \Phi w (r_0 r_1)^{-q} w^{-1}w(r_0 r_1)^{-k} r_1 w'^{-1}\\
%&=& \Phi w (r_0 r_1)^{-k} r_1 w'^{-1}\text{ (since $\Phi w (r_0 r_1)^{-q} w^{-1}w=\Phi w$})\\
%&=& \Phi w' w'^{-1} = \Phi.
%\end{eqnarray*}
\end{proof}

\section{Recursively Enumerable Presentations for the Uniform Tilings}
In this section we give recursively enumerable presentations for each uniform tiling by providing explicit generators for the stabilizer of a specified base flag. In each description, the tiling has universal cover \mc P of Schl\"afli type $\{p,q\}$, and $\Gamma=[p,q]=\langle \rho_{0},\rho_{1},\rho_{2}\rangle$ is the corresponding string C-group. We choose particular words $\beta$ and $\gamma$ in $\Gamma$ which act as translations $t_{1}, t_{2}$ on the base flag $\Phi$ with the following properties: 
\begin{enumerate}
   \item the translation vectors corresponding to $\beta$ and $\gamma$ are linearly independent,
   \item the image of $\Phi$ under either translation $t_{1}$ or $t_{2}$  has minimal distance from $\Phi$ among all possible translates in that direction with respect to the symmetry group of the tiling.
\end{enumerate}
Let $\Psi$ be a flag in the orbit of the base flag. It follows from (\ref{eq:actionsymmetry}) that $\Psi{\beta}$
and $\Psi{\gamma}$ are translates (under the symmetry group of the tiling) of $\Psi$, and therefore $\Psi{\beta^k} = \Psi t_1^k$ and $\Psi{\gamma^{k}} = \Psi t_2^k$ where $t_1$ and $t_2$ are the translations such that $\Psi{\beta} = \Psi t_1$ and $\Psi{\gamma} = \Psi t_2$ respectively. We also choose words $\alpha_{i}$ of the form $W_{f}=w_{f}^{-1}(\rho_{0}\rho_{1})^{p_{f}}w_{f}$ (as in Theorem \ref{t:VFgenerate}) for some face $f$ in the $i$-th transitivity class of polygons under that same translation subgroup of the symmetries of the tiling. Note that if the Schl\"afli type of the cover is $\{p,q\}$ then all words of the type $w_{f}^{-1}(\rho_{0}\rho_{1})^{p}w_{f}$ and
$w_{f}^{-1}(\rho_{1}\rho_{2})^{q}w_{f}$ are trivial and may be omitted. For convenience, we will
use the notation $a=\rho_{0},b=\rho_{1},c=\rho_{2}$, and $w_1^{w_2} = w_2^{-1} w_1 w_2$.

\subsection*{3.6.3.6}

This tiling is covered by the universal tiling $\mc P=\{6,4\}$.
We choose a base flag $\Phi$ containing a hexagon of the tiling (note that all of these lie in a single transitivity class under the symmetry group of $\{6, 4\}$). Then, by Theorem
\ref{t:VFgenerate}, $Stab_{\Gamma(\mc P)}(\Phi)$ is generated only by elements $w_{f}^{-1}(\rho_{0}\rho_{1})^{3}w_{f}$, since the elements $w_{f}^{-1}(\rho_{0}\rho_{1})^{6}w_{f}$ and $w_{v}^{-1}(\rho_{1}\rho_{2})^{4}w_{v}$ are trivial for every $w$. The generating elements are thus obtained as conjugates of elements inducing closed walks around the triangles of the tiling. Note that there are only two classes of triangles under the translation group of the tiling. Let $\alpha_{0}=((ab)^{3})^c, \alpha_{1}={((ab)^{3})}^{cb}$, $\beta=ababacbc, \gamma=abcbabcb$ (Figure \ref{f:3.6.3.6}). Then $\alpha_0$ and $\alpha_1$ correspond to paths around triangles which are not translates of each others.  Lemma \ref{l:notallofthem} now implies that $Stab_{\Gamma(\mc P)}(\Phi)=\langle \alpha_{i}^{\beta^{j}\gamma^{k}}\rangle$ where $i=0,1$ and $j,k\in\ZZ$.
\begin{figure}[h!tbp]
\begin{center}
\includegraphics[height=1.75in]{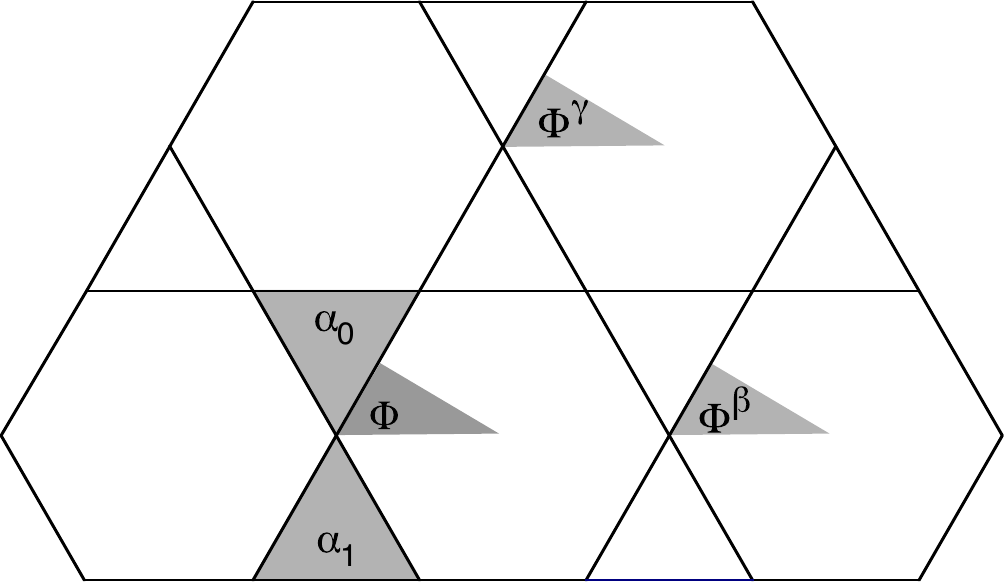}
\caption{The base flag $\Phi$, with the images under the flag action of $\Phi$ by $\beta$ and $\gamma$, as well as the faces traversed by $\alpha_{0}$ and $\alpha_{1}$ for the tiling 3.6.3.6.}
\label{f:3.6.3.6}
\end{center}
\end{figure}%Checked and verified
%By Theorem \ref{t:VFgenerate}) that

\subsection*{4.8.8}
This tiling is covered by the universal tiling $\mc P=\{8,3\}$.
We choose a base flag $\Phi$ containing an edge shared by two octagons of the tiling (note that all of these lie in a single transitivity class). Let $\alpha_{0}=((ab)^{4})^{cb}, \beta=ababcbab, \gamma=cbababab$ (Figure \ref{f:4.8.8}); then $Stab_{\Gamma(\mc P)}(\Phi)=\langle \alpha_{0}^{\beta^{j}\gamma^{k}}\rangle$ where $j,k\in\ZZ$.

\begin{figure}[h!tbp]
\begin{center}
\includegraphics[height=1.75in]{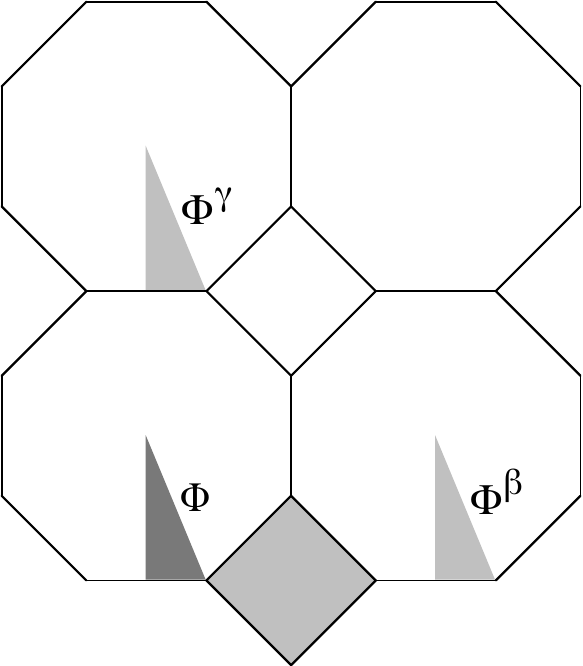}
\caption{The base flag $\Phi$, with the images under the flag action of $\Phi$ by $\beta$ and $\gamma$, as well as the face traversed by $\alpha_{0}$ for the tiling $4.8.8$.}\label{f:4.8.8}
\end{center}
\end{figure}%Checked and verified.

\subsection*{3.3.4.3.4}
This tiling is covered by the universal tiling $\mc P=\{12,5\}$.
We choose a base flag $\Phi$ containing a square such that $\Phi{cbc}$ also contains a square (note that all of these lie in a single transitivity class). Let $\alpha_{0}=(ab)^{4}, \alpha_{1}=((ab)^{3})^{c},\alpha_{2}=\alpha_{0}^{cbc},\alpha_{3}=((ab)^{{3}})^{cbcbc},\alpha_{4}=((ab)^{3}))^{cb},\alpha_{5}=((ab)^{3})^{cbac}, \beta=abcbabcbcb,$ and $\gamma=cabcbacbcbabcb$ (Figure \ref{f:3.3.4.3.4}); then $Stab_{\Gamma(\mc P)}(\Phi)=\langle \alpha_{i}^{\beta^{j}\gamma^{k}}\rangle$ where $i=0,...,5$ and $j,k\in\ZZ$.
\begin{figure}[h!tbp]
\begin{center}
\includegraphics[height=1.75in]{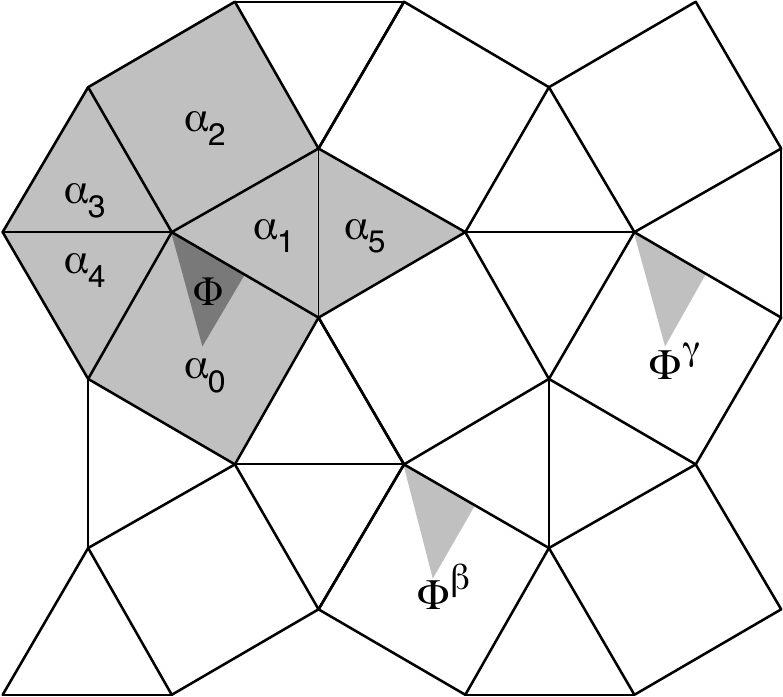}
\caption{The base flag $\Phi$, with the images under the flag action of $\Phi$ by $\beta$ and $\gamma$, as well as the faces traversed by $\alpha_{i}, i=0,...,5$,  for the tiling $3.3.4.3.4$.}\label{f:3.3.4.3.4}
\end{center}
\end{figure}%Checked and verified.

\subsection*{3.3.3.4.4}
This tiling is covered by the universal tiling $\mc P=\{12,5\}$.
We choose a base flag $\Phi$ containing an edge shared by a triangle and a square, and also containing a square of the tiling (as indicated in Figure \ref{f:3.3.3.4.4}). Let $\alpha_{0}=(ab)^{4}, \alpha_{1}=((ab)^{3})^{c},\alpha_{2}=((ab)^{3})^{cbc},\beta=abcb, \gamma=cbab(cb)^{2}ab$; then $Stab_{\Gamma(\mc P)}(\Phi)=\langle \alpha_{i}^{\beta^{j}\gamma^{k}}\rangle$ where $i=0,1,2$  and $j,k\in\ZZ$.
\begin{figure}[h!tbp]
\begin{center}
\includegraphics[height=1.75in]{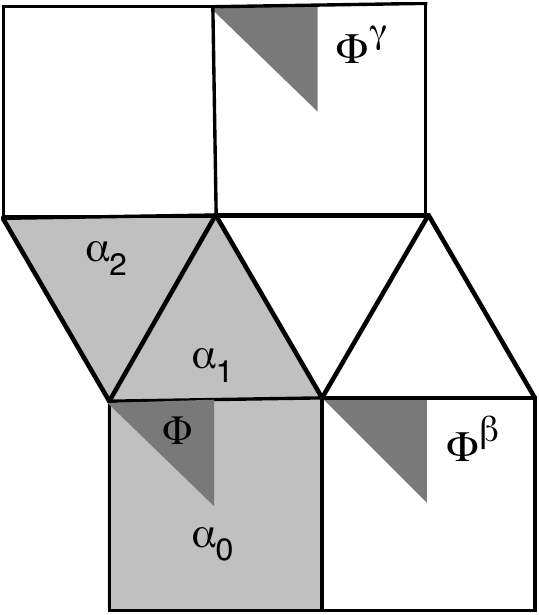}
\caption{The base flag $\Phi$, with the images under the flag action of $\Phi$ by $\beta$ and $\gamma$, as well as the faces traversed by $\alpha_{i}, i=0,...,4$, for the tiling $3.3.3.4.4$.}\label{f:3.3.3.4.4}
\end{center}
\end{figure}%Checked and verified.

\subsection*{3.4.6.4}
This tiling is covered by the universal tiling $\mc P=\{12,4\}$.
We choose a base flag $\Phi$ containing an edge shared by a triangle and a square of the tiling (note that all of these lie in a single transitivity class, see Figure \ref{f:3.4.6.4}), as well as the triangle containing that edge. Let $\alpha_{0}=(ab)^{3}, \alpha_{1}=((ab)^{4})^{cba},\alpha_{2}=((ab)^{4})^{cb},\alpha_{3}=((ab)^{4})^{c}, \alpha_{4}=((ab)^{6})^{cbc},\alpha_{5}=((ab)^{3})^{cbabc},\beta=cbabcbcbabcbab, \gamma= caba(bc)^{2}babcab $; then $Stab_{\Gamma(\mc P)}(\Phi)=\langle \alpha_{i}^{\beta^{j}\gamma^{k}}\rangle$ where $i=0,...5$ and $j,k\in\ZZ$.
\begin{figure}[h!tbp]
\begin{center}
\includegraphics[height=2.25in]{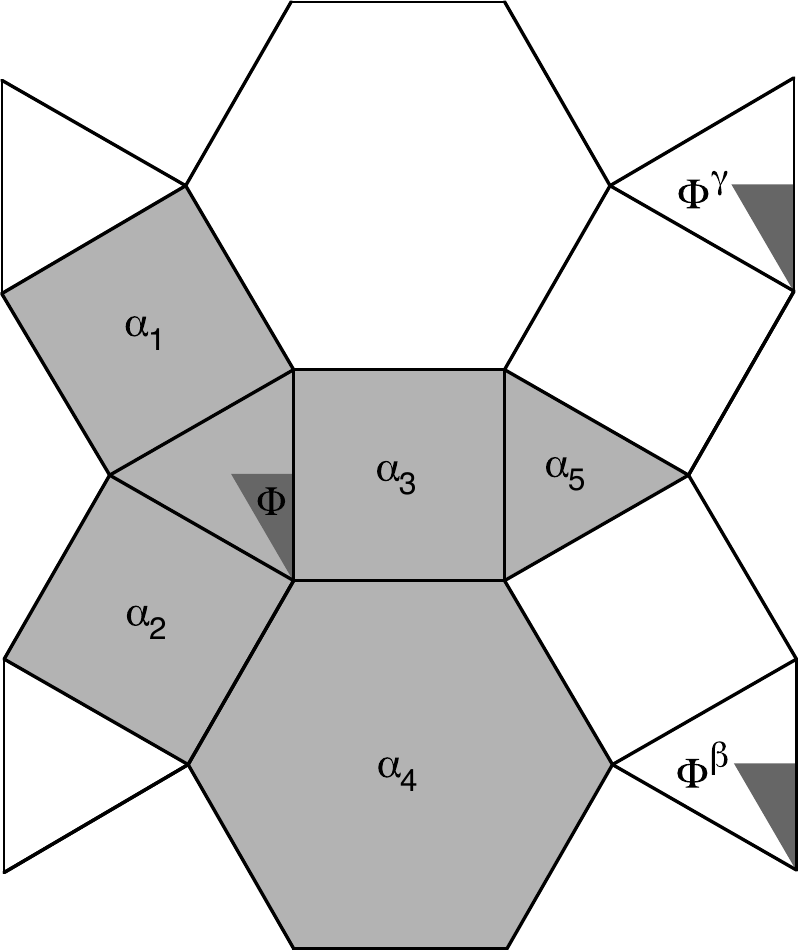}
\caption{The base flag $\Phi$, with the images under the flag action of $\Phi$ by $\beta$ and $\gamma$, as well as the faces traversed by $\alpha_{i}, i=0,\ldots,5$, for the tiling 3.4.6.4.}\label{f:3.4.6.4}
\end{center}
\end{figure}%Checked and verified.

\subsection*{3.3.3.3.6}
This tiling is covered by the universal tiling $\mc P=\{6,5\}$.
We choose a base flag $\Phi$ containing a triangle and an edge in a hexagon of the tiling (note these flags lie in two transitivity classes since there is no mirror symmetry of the tiling; see Figure \ref{f:3.3.3.3.6}). Let $\alpha_{0}=(ab)^{3}, \alpha_{1}=\alpha_{0}^{cbacbc},\alpha_{2}=\alpha_{0}^{cbc},\alpha_{3}=\alpha_{0}^{cbcb},\alpha_{4}=\alpha_{0}^{cb},\alpha_{5}=\alpha_{0}^{cba},\alpha_{6}=\alpha_{0}^{cbcba},\alpha_{7}=\alpha_{0}^{cbca},\beta=ab(cb)^{3}(abcb)^{2}cb, \gamma=ca(ba)^{2}(bc)^{2}ab$; then $Stab_{\Gamma(\mc P)}(\Phi)=\langle \alpha_{i}^{\beta^{j}\gamma^{k}}\rangle$ where $i=0,\ldots,7$ and $j,k\in\ZZ$.
\begin{figure}[h!tbp]
\begin{center}
\includegraphics[height=2.5in]{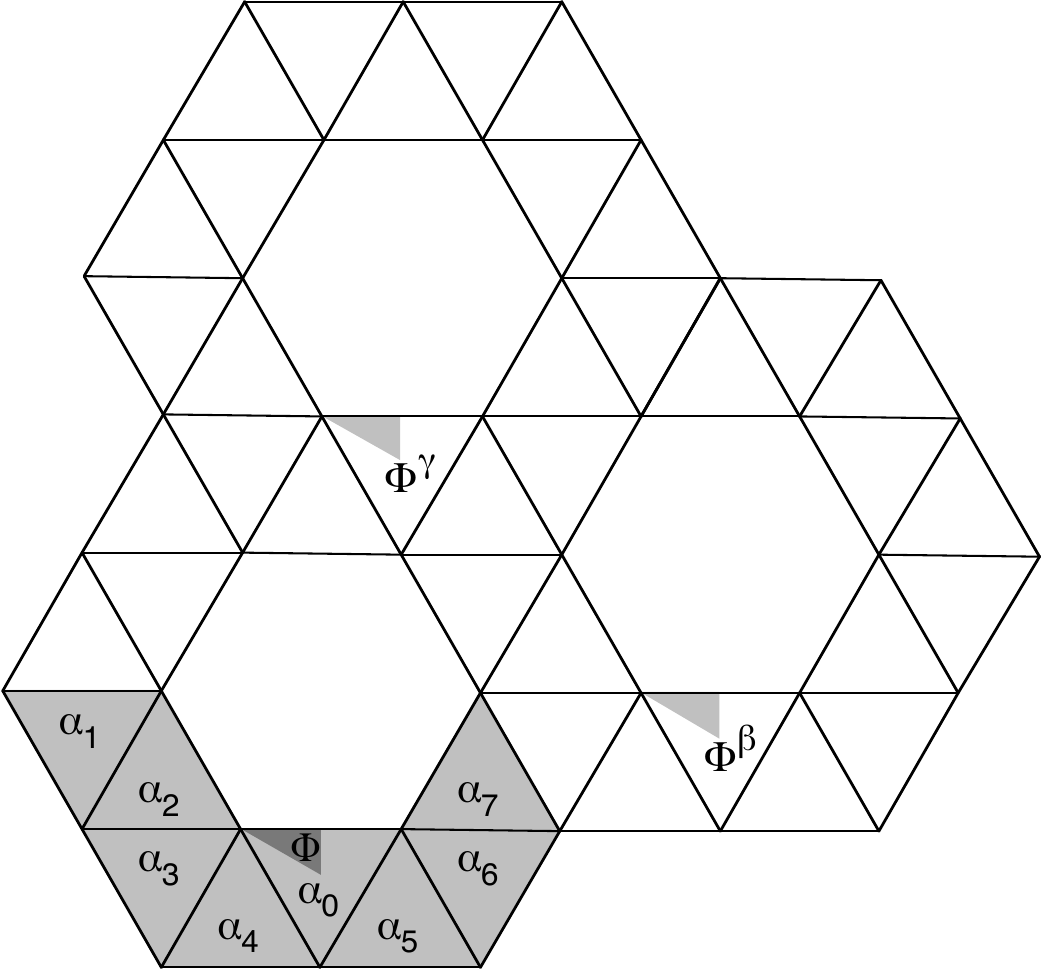}
\caption{The base flag $\Phi$, with the images under the flag action of $\Phi$ by $\beta$ and $\gamma$, as well as the faces traversed by $\alpha_{i}, i=0,...,7$, for the tiling $3.3.3.3.6$.}\label{f:3.3.3.3.6}
\end{center}
\end{figure}%Checked and verified.

\subsection*{3.12.12}
This tiling is covered by the universal tiling $\mc P=\{12,3\}$.
We choose a base flag $\Phi$ containing an edge shared by two dodecagons of the tiling (note that all of these lie in a single transitivity class). Let $\alpha_{0}=((ab)^{3})^{cb}, \alpha_{1}=((ab)^{3})^{cbabab},\beta=(bcba)^{2}(ba)^{2}, \gamma=(ba)^{2}(bcba)^{2}$ (Figure \ref{f:3.12.12}); then $Stab_{\Gamma(\mc P)}(\Phi)=\langle \alpha_{i}^{\beta^{j}\gamma^{k}}\rangle$ where $i=0,1$ and $j,k\in\ZZ$.
\begin{figure}[h!tbp]
\begin{center}
\includegraphics[height=1.75in]{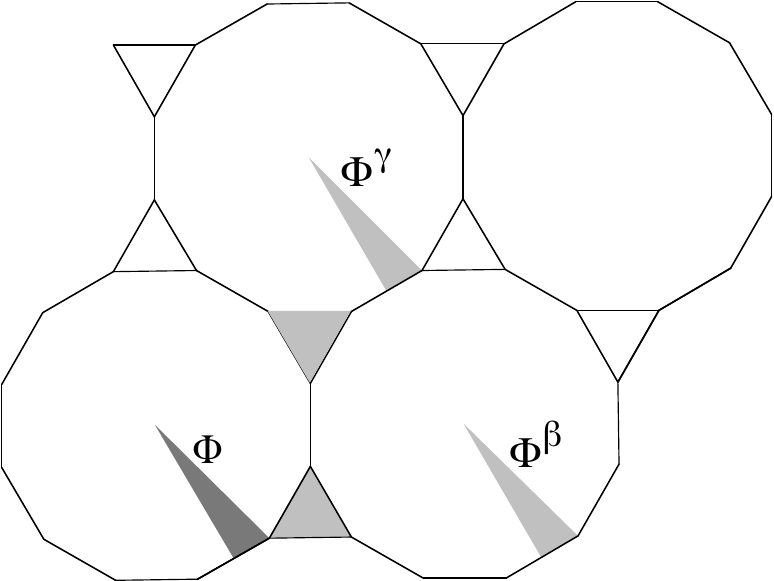}
\caption{The base flag $\Phi$, with the images under the flag action of $\Phi$ by $\beta$ and $\gamma$, as well as the faces traversed by $\alpha_{0}$ and $\alpha_{1}$ for the tiling $3.12^{2}$.}\label{f:3.12.12}
\end{center}
\end{figure}%Checked and verified.

\subsection*{4.6.12}
This tiling is also covered by the universal tiling $\mc P=\{12,3\}$.
We choose a base flag $\Phi$ containing a dodecagon and an edge of a hexagon of the tiling (note that all of these lie in a single transitivity class). Let $\alpha_{0}=((ab)^{4})^{cbabab}, \alpha_{1}=((ab)^{6})^{cbab},\alpha_{2}=((ab)^{4})^{cb},\alpha_{3}=((ab)^{6})^{c},\alpha_{4}=((ab)^{4})^{cba},\beta=(ab)^{3}(cbab)^{2}ab, \gamma=(ab)^{5}cbabcb$ (Figure \ref{f:4.6.12}); then $Stab_{\Gamma(\mc P)}(\Phi)=\langle \alpha_{i}^{\beta^{j}\gamma^{k}}\rangle$ where $i=0,\ldots,4$ and $j,k\in\ZZ$.
\begin{figure}[h!tbp]
\begin{center}
\includegraphics[height=2.5in]{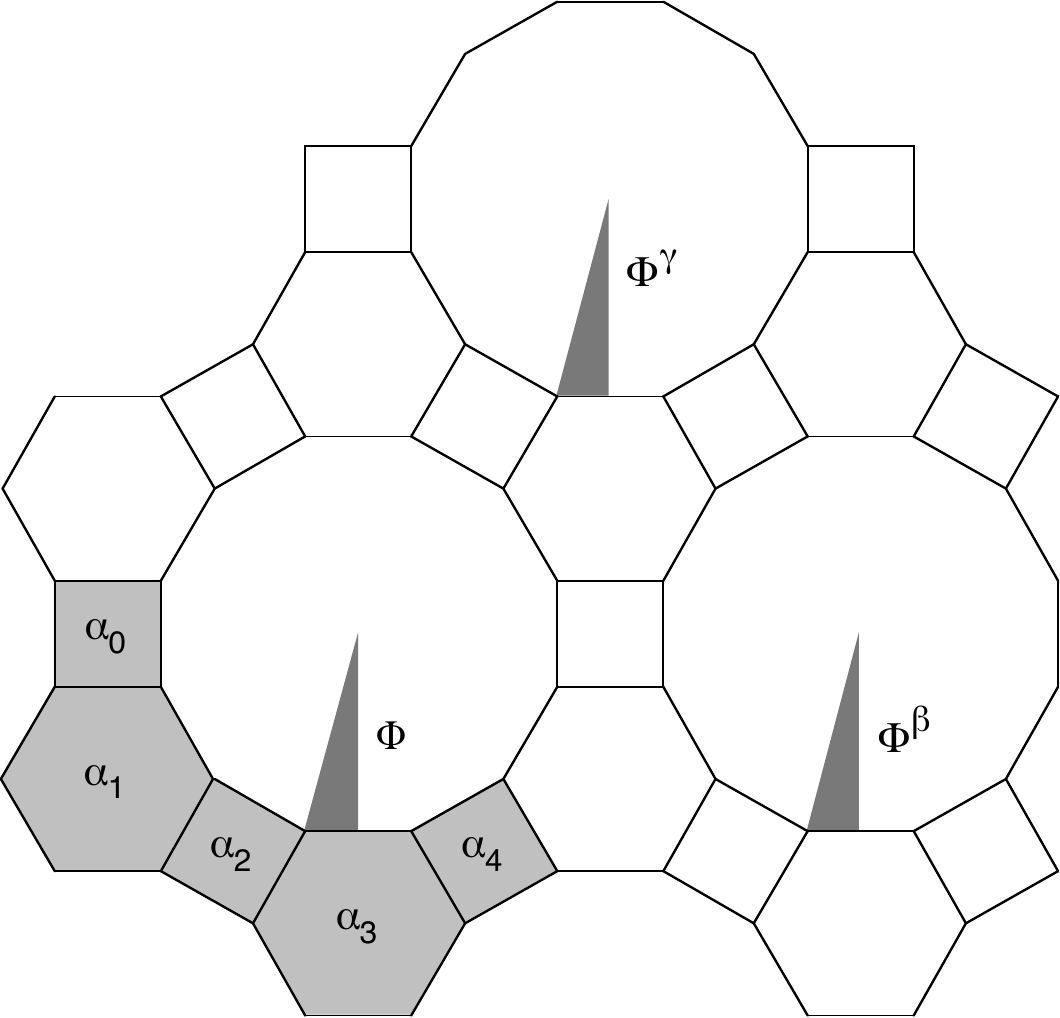}
\caption{The base flag $\Phi$, with the images under the flag action of $\Phi$ by $\beta$ and $\gamma$, as well as the faces traversed by $\alpha_{i}, i=0,...,4$, for the tiling 4.6.12.}\label{f:4.6.12}
\end{center}
\end{figure}%Checked and verified.

\section{Conclusion}
Closed walks, spanning trees and the flag graph have been previously used for different purposes related to stabilizers of flags (see for example \cite{OrbPelWei}, \cite{PisZit}). In \cite[Theorem 2F4]{McMSch02} McMullen and Schulte interpret the elements in the stabilizer of a flag of a regular polytope $\mc P$ as closed walks on the graph determined by the vertices and edges of $\mc P$ (as opposed to Theorem \ref{T:SpanningTrees}, where we use $\mc{GF}(\mc P)$ instead). This is used to determine a generating set for the stabilizer of a flag for the infinite polyhedron $\{\infty, 3\}^{(b)}$ in \cite[Section 7E]{McMSch02}.

The current work is motivated by three related goals. The first is to better understand the relationship between the geometry of classically studied polyhedra and Hartley's quotient represenation. The second is to begin to lay the groundwork for the study of new classes of non-regular polytopes. Finally, we hope to develop some of the tools necessary to utilize abstract polytopes to resolve some of the outstanding questions in the study of tilings and polyhedra. For example,  in 1981 Gr\"unbaum, Miller and Shephard posited a complete classification of generalized uniform tilings admitting the possibility of non-convex planar star polygons and apeirogons as faces \cite{GruMilShe81}. To date, no proof of the completeness of this enumeration has appeared. Is it possible to analyze such tilings from within a framework of abstract polyhedra to verify the enumeration?

------------------------------------------------------------------------------------------

\bibliographystyle{amsalpha}
\bibliography{../../ResearchBibliography}

 \end{document}